%% file: Squareful_numbers_revision_ANT_June_2011.tex
\documentclass[11pt,twoside,a4paper]{amsart}
\usepackage[english]{babel}
\usepackage{url}
\usepackage{a4wide}
\usepackage{amssymb}
\usepackage[all]{xy}
\usepackage{pstricks}
\usepackage{multido}
\allowdisplaybreaks
\theoremstyle{plain}
\newtheorem{theo}{Theorem}[section]
\newtheorem{lem}[theo]{Lemma}
\newtheorem{cor}[theo]{Corollary}
\newtheorem{prop}[theo]{Proposition}
\theoremstyle{definition}
\newtheorem{definition}{Definition}
\theoremstyle{remark}
\newtheorem{remark}[theo]{Remark}

\DeclareMathOperator{\lcm}{lcm}

\DeclareMathOperator{\sgn}{sgn}
\DeclareMathOperator{\Val}{Val}

\DeclareMathOperator{\Pic}{Pic}

\let\leq\leqslant

\let\geq\geqslant
\newcommand{\R}{\ensuremath{\mathbf{R}}}
\newcommand{\Z}{\ensuremath{\mathbf{Z}}}
\newcommand{\Q}{\ensuremath{\mathbf{Q}}}
 
\title[Squareful numbers in hyperplanes]{Squareful numbers in hyperplanes}
\author{Karl Van Valckenborgh}
\address{K.U.Leuven, Department of Mathematics, Celestijnenlaan 200B, 3001 Leuven, Belgium}
\email{karl.vanvalckenborgh@wis.kuleuven.be}
\date{\today}

\begin{document}

\begin{abstract}
Let $n \geqslant 4$. In this article, we will determine the asymptotic behaviour of the size of the set of integral points $(a_0:\ldots :a_n)$ on the hyperplane $\sum_{i=0}^{n}X_i=0$ in $\mathbf{P}^n$ such that $a_i$ is squareful (an integer $a$ is called squareful if the exponent of each prime divisor of $a$ is at least two) and $|a_i|\leqslant B$ for each $i \in \{0,\ldots,n\}$, when $B$ goes to infinity. For this, we will use the classical Hardy-Littlewood method. The result obtained supports a possible generalization of the Batyrev-Manin program to Fano orbifolds. 
\end{abstract}
\maketitle
\section{Introduction}\label{s:intro}
The problem we consider can be related to a question Campana posed concerning rational points on orbifolds. A good overview is given for example in \cite{abramovich}, \cite{poonen} or \cite{Campana}. Examining the orbifold $(\mathbf{P}^1,\Delta)$ with $\Q$-divisor $\Delta= 1/2\cdot[0]+1/2\cdot[1]+1/2\cdot[\infty]$, it is explained for example in \cite{poonen} why it is reasonable to expect that the set $\{(a_1,a_2,a_3) \in \Z^3 \mid a_1+a_2=a_3,\ a_1,a_2,a_3\ \text{ are squareful, }\max\{|a_1|,|a_2|,|a_3|\}\leqslant B,\ \gcd(a_1,a_2,a_3)=1\}$ will asymptotically behave as $C\cdot B^{1/2}$ as $B$ tends to infinity. 

Since this question turns out to be too difficult at the moment, we generalize to a higher-dimensional analogue $(\mathbf{P}^{n-1},\Delta)$, where now $\Delta$ is the $\Q$-divisor $\Delta=1/2\cdot[H_{0}]+\dots +1/2\cdot[H_{n}]$ with $H_i$ the hyperplane defined by $X_i=0$ for $i\in \{0,\ldots,n-1\}$ and $H_{n}$ defined by $X_0+\dots + X_{n-1}=0$. Analogously as in the one-dimensional case, a point $P=(a_0:\ldots:a_{n-1}) \in \mathbf{P}^{n-1}(\Q)$ (we assume $a_0,\ldots,a_{n-1} \in \Z$ and $\gcd(a_0,\ldots,a_{n-1})=1$) will be called a rational point in Campana's sense on $(\mathbf{P}^{n-1},\Delta)$ if for every $i \in \{0,\ldots, n\}$ and every prime $p$ for which the reduction of $P$ is contained in the reduction of $H_i$ modulo $p$, we have that $i_p(P,H_i)\geqslant 2$, where $i_p(P,H_i)$ denotes the intersection number of $P$ and $H_i$ above the prime $p$. These conditions will be satisfied if $a_i$ is squareful for every $i \in \{0,\ldots,n-1\}$ and if $\sum_{i=0}^{n-1}a_i$ is also squareful. We denote the set of all such rational points by $(\mathbf{P}^{n-1},\Delta)(\Q)$; the set of the points $P\in (\mathbf{P}^{n-1},\Delta)(\Q)$ of bounded height (using the height function $H(x_0:\ldots:x_{n-1})=\max\{|x_0|,\ldots,|x_{n-1}|,|\sum_{i=0}^{n-1}x_i|\}$) is denoted by $(\mathbf{P}^{n-1},\Delta)(\Q)_{\leqslant B}$. 

Defining the canonical divisor of the orbifold $(\mathbf{P}^{n-1},\Delta)$ as $K_{(\mathbf{P}^{n-1},\Delta)}=K_{\mathbf{P}^{n-1}}+\Delta$, we have that $K_{(\mathbf{P}^{n-1},\Delta)}\sim (-(n-1)/2)\cdot H$ in $\Pic(\mathbf{P}^{n-1})_{\Q}$, where $H$ is the hyperplane class of $\mathbf{P}^{n-1}$. Since the height function we use is associated to $H$, a very na\"ive generalization of Manin's conjecture would predict that $\#(\mathbf{P}^{n-1},\Delta)(\Q)_{\leqslant B} \sim C\cdot B^{(n-1)/2}$ for some constant $C>0$, as $B$ tends to infinity. Our main goal is to prove the following theorem.
\begin{theo}\label{theo:main_theorem}
For $n\geqslant 4$, ther exists a $\delta>0$ so that 
$$
\#(\mathbf{P}^{n-1},\Delta)(\Q)_{\leqslant B}= C\cdot B^{(n-1)/2}+O\left(B^{(n-1)/2-\delta}\right)
$$
for some constant $C>0$. 
\end{theo}
Moreover, in Section \ref{sec:orbi} we will give an explicit description of the constant $C$ and examine the distribution of rational points on the orbifold $(\mathbf{P}^{n-1},\Delta)$.

I would like to express my gratitude to my advisor professor Emmanuel Peyre for the many helpful conversations concerning these subjects, and to the section of Number Theory of the department of Mathematics of the University of Bristol (in particular professor Tim Browning and professor Trevor Wooley) for the useful tips concerning the circle method. Also, I would like to thank the reviewers for some interesting comments which led to substantial improvements of the result and the presentation. 
\section{Description of the proof}
Throughout the article, we will use the following notations. 

We will denote the $(n+1)$-tuple $(x_0,\ldots,x_n) \in A^{n+1}$ for any ring $A$ by $\underline{x}$. For the nonzero integers we use the notation $\Z_0$, i.e. $\Z_0=\Z\setminus\{0\}$. If there exists a constant $C>0$ such that $|f(x)| \leqslant C g(x)$ for real-valued functions $f$ and $g$ with $g$ only taking positive values, we write $f(x) \ll g(x)$ or $f(x)=O(g(x))$. If $C$ depends on other parameters, this will be denoted explicitly when this dependence is important for the computations. We will write $f(x) \sim g(x)$ if $f(x)/g(x)$ tends to one if $x$ goes to infinity. Also, we allow the small positive constant $\varepsilon$ to take different values at different points of the arguments. Finally, for any $\alpha \in \R$ we will write $e(\alpha)=\exp(2\pi i \alpha)$. 

To prove Theorem \ref{theo:main_theorem}, we first restrict ourselves to the set of points $(a_0:\ldots:a_{n-1})\in (\mathbf{P}^{n-1},\Delta)(\Q)$ for which $a_{i}\neq 0$ for each $i\in \{0,\ldots,n-1\}$ and $\sum_{i=0}^{n-1}a_i\neq 0$. We denote this subset by $(\mathbf{P}^{n-1},\Delta)(\Q)^+$. Also, by $(\mathbf{P}^{n-1},\Delta)(\Q)_{\leqslant B}^+$ we indicate the intersection of $(\mathbf{P}^{n-1},\Delta)(\Q)^+$ with $(\mathbf{P}^{n-1},\Delta)(\Q)_{\leqslant B}$.

From the definition of $(\mathbf{P}^{n-1},\Delta)(\Q)$, it follows that we can identify $(\mathbf{P}^{n-1},\Delta)(\Q)_{\leqslant B}^+$ with the set
\begin{multline*}
\Bigl\{(a_0:\ldots:a_n)\in H(\Q): a_i\in \Z_0\text{ and }a_i\text{ is squareful},\ \gcd(a_0,\ldots,a_n)=1,\ \max_{0\leqslant i\leqslant n}|a_i|\leqslant B\Bigr\},
\end{multline*}
where $H\subset \mathbf{P}^n$ is the hyperplane defined by $X_0+\cdots+X_n=0$.

Since a squareful integer can be written \lq uniquely\rq\ as $x^2y^3$ where $y$ is squarefree (this representation is unique up to the sign of $x$), the latter set in turn corresponds to 
\begin{multline}\label{eq:orbifold} 
\Bigl\{(x_0^2y_0^3:\ldots:x_n^2y_n^3)\in H(\Q): x_i,y_i \in \Z_0\text{ and } y_i \text{ is squarefree},\\ \gcd(x_0y_0,\ldots,x_ny_n)=1,\ \max_{0\leqslant i\leqslant n}|x_i^2y_i^3|\leqslant B  \Bigr\}.
\end{multline}
We also introduce the following definition. 
\begin{definition}\label{different sets}
We define $M(B)$ as the set 
\begin{equation*}
\Biggl\{(\underline{x},\underline{y})\in \Z_0^{2n+2}: \sum_{i=0}^n x_i^2y_i^3=0,\ \gcd(x_0y_0,\ldots,x_ny_n)=1,\ \max_{0\leqslant i\leqslant n}|x_i^2y_i^3|\leqslant B,\ \prod_{i=0}^n\mu^2(|y_i|)\neq0 \Biggr\}.
\end{equation*}
Also, by $M_{\underline{a},t}(B)$ we denote the set
\begin{equation*}
\Biggl\{(\underline{x},\underline{y})\in \Z_0^{2n+2}: \sum_{i=0}^n a_ix_i^2y_i^3=t,\  \max_{0\leqslant i\leqslant n}|a_ix_i^2y_i^3|\leqslant B,\ \prod_{i=0}^n\mu^2(|y_i|)\neq0 \Biggr\},
\end{equation*}
where $a_{0},\ldots,a_{n},t \in \Z$ are fixed, $\gcd(a_{0},\ldots,a_n)=1$ and $\prod_{i=0}^{n}a_{i}\neq 0$. (Note that for any integer $y\in \Z$, the condition $\mu^2(|y_i|)\neq 0$ means that $y_i$ is squarefree.)
\end{definition}

As a first step in the proof, we will use the classical Hardy-Littlewood circle method to determine an expression for the cardinality of the set $M_{\underline{a},t}(B)$. From this result, we will derive an asymptotic formula for $\#M(B)$: indeed, we see that $M(B)$ is a subset of $M_{(1,\ldots,1),0}(B)$, with the additional gcd condition $\gcd(x_0y_0,\ldots,x_ny_n)=1$ on the solutions. We will take this gcd condition into account using an adapted version of the M\"obius inversion.

Identifying $(\mathbf{P}^{n-1},\Delta)(\Q)_{\leqslant B}^+$ with \eqref{eq:orbifold}, it readily follows that 
$$
\#(\mathbf{P}^{n-1},\Delta)(\Q)_{\leqslant B}^+=\frac{1}{2^{n+2}}\#M(B),
$$
which implies that an asymptotic formula for $\#M(B)$ induces an asymptotic formula for $\#(\mathbf{P}^{n-1},\Delta)(\Q)_{\leqslant B}^+$. 

Finally, we will explain why this result suffices to prove Theorem \ref{theo:main_theorem}.

\section{Calculating $\#M_{\underline{a},t}(B)$}
Let us first fix the framework of the circle method. 

Let $T$ be $\R/\Z$. For $0<\Delta \leqslant 1$ and $P\geqslant1$ (we always suppose $B \geqslant 1$), we define $\mathfrak{M}(\Delta,q,a)$ as the image in $T$ of $\{\alpha \in \R \ |\ |\alpha-a/q|<P^{\Delta-2}\}$ with $a,q \in \Z$ and 
$$
\mathfrak{M}(\Delta)=\mathop{\bigcup_{{1\leqslant a\leqslant q\leqslant P^{\Delta}}}}_{\gcd(a,q)=1}\mathfrak{M}(\Delta,q,a).
$$
We call $\mathfrak{M}(\Delta)$ the union of the \textit{major arcs} and $T \setminus \mathfrak{M}(\Delta)=\mathfrak{m}(\Delta)$ the union of the \textit{minor arcs}. We shall clarify the constraint on the constant $\Delta$ and the dependence of $P$ on $B$ in Proposition \ref{prop:major} and Theorem \ref{theorem_major}. 
  
The circle method calculates $\#M_{\underline{a},t}(B)$ by integrating an exponential sum over $T$, namely
\begin{equation}\label{eq:M_{a,t}(B)}
\#M_{\underline{a},t}(B)=\int_{T}\mathop{\sum_{1\leqslant |a_{i}x_{i}^2y_{i}^3|\leqslant B}}_{i=0,\ldots,n}\left(\prod_{i=0}^{n}\mu^2(|y_{i}|)\right)e(\alpha f(\underline{x},\underline{y}))d\alpha,
\end{equation}
where $f(\underline{x},\underline{y})=\sum_{i=0}^{n}a_{i}x_{i}^2y_{i}^3-t$. We will denote the integrand of this integral with $E(\alpha)$ and also 
$$
S_{i}(\alpha)= \sum_{1\leqslant |a_{i}x^2y^3|\leqslant B}\mu^2(|y|)e(\alpha a_{i}x^2y^3).
$$
Therefore,
$$
E(\alpha)=e(-\alpha t)\prod_{i=0}^{n}S_i(\alpha).
$$
As usual, the integral over $\mathfrak{M}(\Delta)$ will provide the main term while the integral over $\mathfrak{m}(\Delta)$ will only contribute to the error term. 
\input{Major_arcs_revision_ANT_June_2011}
\subsection{Minor arcs}
The goal of this section is to prove the following theorem.
\begin{theo}{\label{theo:minor}} For $n\geqslant 4$, there exists a $\delta>0$ so that 
$$
\int_{\mathfrak{m}(\Delta)}E(\alpha)d\alpha=O\left(B^{(n-1)/2-\delta}\right). 
$$
\end{theo}
To treat the integral over the minor arcs, we will not fix $\underline{y}$ but examine the whole equation at once.
Recall that 
$$
E(\alpha)=e(-\alpha t)\prod_{i=0}^{n}S_{i}(\alpha)=e(-\alpha t)\prod_{i=0}^{n} \sum_{1\leqslant|a_{i}x^2y^3|\leqslant B}\mu^2(|y|) e(\alpha a_{i}x^2y^3).
$$
Using H\"older's inequality repeatedly, we get for $n \geqslant 4$,  
\begin{equation}
\left|\int_{\mathfrak{m}(\Delta)}E(\alpha)d\alpha \right| \leqslant 
\sup_{\alpha \in \mathfrak{m}(\Delta)}(|S_{0}(\alpha)|\cdots |S_{n-4}(\alpha)|)\max_{j=n-3,\ldots,n}\int_{0}^{1}|S_{j}(\alpha)|^{4}d\alpha.\label{minoren}
\end{equation}
To obtain a good upper bound of this expression, we first examine $\int_{0}^{1}|S_{j}(\alpha)|^{4}d\alpha$.
\begin{lem}\label{lemma} For any $\varepsilon >0$, we have 
$$\int_{0}^{1}|S_{j}(\alpha)|^4d\alpha \ll_{\varepsilon} B^{1+\varepsilon}.$$ 
\end{lem}
\begin{proof}
From now on, we will concentrate on the part of the sum where the variables are positive. This will suffice to prove the theorem because of the symmetry.

Let 
$$
 S_{Y}(\alpha)= \sum_{Y<y\leqslant 2Y}\mu^2(y)\sum_{1\leqslant x\leqslant B_{a_{j},y}} e(\alpha a_{j}x^2y^3)
$$
be the contribution to $S_{j}(\alpha)$ for $Y<y \leqslant 2Y$ and squarefree. Using Cauchy inequality, it follows that 
\begin{align*}
\int_{0}^{1}|S_{Y}(\alpha)|^4 d\alpha & \ll Y \int_{0}^{1}|S_{Y}(\alpha)|^2 \sum_{Y<y\leqslant 2Y}\mu^2(y)\left|\sum_{1\leqslant x\leqslant B_{a_{j},y}} e(\alpha a_{j}x^2y^3)\right|^2 d\alpha \\
& \ll Y\sum_{Y<y_{1},y_{2},y_{3}\leqslant 2Y}\mathop{\mathop{\sum_{1\leqslant x_1\leqslant B_{a_j,y_1}}}_{1\leqslant x_2\leqslant B_{a_j,y_2}}}_{1\leqslant x_3,x_4\leqslant B_{a_j,y_3}}\int_{0}^{1}e(\alpha a_{j}G(\underline{x},\underline{y}))d\alpha \\
& \leqslant Y\cdot \#Z(Y,B),
\end{align*}
with $G(\underline{x},\underline{y})=y_{3}^3(x_{4}^2-x_{3}^{2})+x_{1}^2y_{1}^3-x_{2}^2y_{2}^3$ and $Z(Y,B)= \{ (\underline{x},\underline{y}) \in \Z_0^{7}:y_{3}^3(x_{3}^2-x_{4}^2)=x_{1}^2y_{1}^3-x_{2}^2y_{2}^3,\ 1\leqslant x_i< B_{Y},\ Y<y_{j}\leqslant2Y\}$, where $B_Y=(B/Y^3)^{1/2}$. 

%
If we make a distinction between solutions $(\underline{x},\underline{y})\in \Z_0^7$ of 
$G(\underline{x},\underline{y})=0$ for which $x_1^2y_1^3-x_{2}^2y_{2}^3=0$ or not, it follows that both sets contain $O(Y^{-1}\cdot B^{1+\varepsilon})$ solutions. Hence, we conclude that $\#Z(Y,B)\ll_{\varepsilon} Y^{-1}\cdot B^{1+\varepsilon}$ and thus, 
$$
\int_{0}^{1}|S_{Y}(\alpha)|^4 d\alpha \ll_{\varepsilon}B^{1+\varepsilon}.
$$  
Summing over all intervals $(Y,2Y]$ with $Y=2^k\ll B^{1/3}$ and applying Cauchy's inequality twice on $|S_j(\alpha)|^4=|\sum_{Y=2^k\ll B^{1/3}}S_Y(\alpha)|^4$, we get
$$
\int_{0}^1 |S_j(\alpha)|^4d\alpha \ll B^{3\varepsilon'}\sum_{Y=2^k\ll B^{1/3}} \int_{0}^{1}|S_{Y}(\alpha)|^4 d\alpha \ll B^{3\varepsilon'}\sum_{Y=2^k\ll B^{1/3}} B^{1+\varepsilon}=B^{1+\varepsilon''},
$$  
which completes the proof.
\end{proof}
\begin{remark}\label{speciaal_geval} Recalling the expression for $\#M_{\underline{a},t}(B)$ in \eqref{eq:M_{a,t}(B)} and putting $n=3$, $\underline{a}=(1,1,1,1)$ and $t=0$, this lemma implies that the equation $n_1+n_2=n_3+n_4$, where $n_i$ is squareful and $1\leqslant |n_{i}|\leqslant B$ for each $i \in \{1,2,3,4\}$, has $O(B^{1+\varepsilon})$ solutions. 
\end{remark}
In order to handle the first part of \eqref{minoren}, namely $\sup_{\alpha \in \mathfrak{m}(\Delta)}(|S_{0}(\alpha)|\cdots |S_{n-4}(\alpha)|)$, we will prove the following proposition.
\begin{prop}\label{kleinste} Let $\alpha \in \mathfrak{m}(\Delta)$. Then there exists a $\delta>0$ such that 
$$
\left|S_{i}(\alpha)\right| \ll B^{1/2-\delta}. 
$$
\end{prop}
\begin{proof}
Let $\psi>0$. We may henceforth assume that $|a_i| \leqslant B^{\psi}$, since otherwise the trivial upper bound yields
\begin{equation*}
\left|S_i(\alpha)\right| \leqslant \sum_{y=1}^{\infty} \sqrt{\frac{B}{a_iy^3}}\ll B^{(1-\psi)/2},
\end{equation*}
which is satisfactory. Similarly, we may assume that $y\leqslant B^{\psi}$ in $S_i(\alpha)$. Thus, we have 
\begin{equation*}
\left| S_i(\alpha)\right| \ll B^{(1-\psi)/2}+\sum_{y\leqslant B^{\psi}}\mu^2(y)\left|T_{\underline{y}}(\alpha)\right|,
\end{equation*}
with, if we set $X=\sqrt{B/(|a_i|y^3)}$, 
\begin{equation*}
T_{\underline{y}}(\alpha)=\sum_{x\leqslant X}e(\alpha a_iy^3x^2).
\end{equation*} 
Since $|a_i|y^3x^2\leqslant B$, we know in particular that $X\geqslant B^{1/2-2\psi}$. Using the usual squaring and differencing approach (see for example \cite[Chapter 3]{dav}), we obtain that
\begin{align*}
\left|T_{\underline{y}}(\alpha)\right|^2&\leqslant \sum_{|h|\leqslant X} \left|\sum_{\substack{x\\x,x+h\leqslant X}}e(2\alpha a_iy^3hx)\right|\\
&\ll \sum_{|h|\leqslant X} \min \{X, \|2\alpha a_iy^3h \|^{-1}\} \\
&\ll X+B^{\varepsilon}\cdot \sum_{y\leqslant Y}\min \{X,\|\alpha y\|^{-1}\},
\end{align*}
where $Y=2|a_i|y^3X$ and $\|a\| = \min\{|\beta| \in \R:\beta \equiv a \bmod 1\}$ for any real number $a$.

In order to estimate the sum over $y$, we will use the following lemma.
\begin{lem}[Separation lemma]\label{lem:separation}  Let $P,Q\geqslant1$ be reals, $\alpha \in T$ and $a,q\in \Z$ with $\gcd(a,q)=1$ and such that $|\alpha -a/q|<q^{-2}$. Then
$$
\sum_{x\leqslant P}\min \left\{ \frac{PQ}{x},\|\alpha x\|^{-1}\right\} \ll PQ\left(q^{-1}+Q^{-1}+q(PQ)^{-1}\right)\log(2qP).
$$  
\end{lem}
\begin{proof}
A full proof is given in \cite[Lemma 2.2]{Vaughan}.
\end{proof}
Choosing $P=Y$ and $Q=X$, Lemma \ref{lem:separation} implies
\begin{align*}
\left|T_{\underline{y}}(\alpha)\right|^2&\ll X+XYB^{\varepsilon}\left(\frac{1}{q}+\frac{1}{X}+\frac{q}{XY}\right)\\
&\ll XY B^{2\varepsilon}\left(\frac{1}{q}+\frac{1}{X}+\frac{q}{XY}\right)\\
&\ll B^{1+2\varepsilon}\left(\frac{1}{q}+ B^{2\psi-1/2}\right)+qB^{2\varepsilon},
\end{align*}
since $X\leqslant Y$ and $XY=2|a_i|y^3X^2=2B$. Hence,
\begin{equation}\label{begrenzing:S_i(a)}
\left|S_i(\alpha)\right| \ll B^{1/2-2\psi}+ B^{1/2+\varepsilon+\psi}\left(\frac{1}{\sqrt{q}}+B^{\psi-1/4}\right)+\sqrt{q}B^{\varepsilon+\psi}.
\end{equation}
According to Dirichlet, we can find $a,q \in \Z$ with $\gcd(a,q)=1$ and $q\leqslant B^{(2-\Delta)/4}$ such that $|\alpha q -a|<1/B^{(2-\Delta)/4}=B^{(\Delta-2)/4}$. (Note we also have $|\alpha-a/q|<1/q^2$.) Furthermore, it is necessary that $q>B^{\Delta/2}$: otherwise, we would have that $\alpha \in \mathfrak{M}(\Delta)$. With these boundaries for $q$ in \eqref{begrenzing:S_i(a)}, a suitable small choice for $\psi$ in terms of $\Delta$ leads to the statement. 
\end{proof}
We are now able to prove Theorem \ref{theo:minor}.
\begin{proof}[Proof of Theorem \ref{theo:minor}]
Combining Proposition \ref{kleinste} and Lemma \ref{lemma} in \eqref{minoren}, we obtain
\begin{align*}
\left|\int_{\mathfrak{m}(\Delta)}E(\alpha)d\alpha\right| &\ll B^{(1/2-\delta)(n-3)}\cdot B^{1+\varepsilon}\\
 & \leqslant B^{(n-1)/2-\delta+\varepsilon} <B^{(n-1)/2}
\end{align*}
for any $0<\varepsilon<\delta$. 
\end{proof}
\input{towards_the_main_problem_revision_ANT_June_2011}

\input{adelic_distribution_revision_version_ANT_June_2011}
\bibliographystyle{alpha}
\bibliography{arcs}
\end{document}

%% file: Major_arcs_revision_ANT_June_2011.tex
\subsection{Major arcs}\label{subsection:major}
We refer to \cite[\S 5]{schmidt} or \cite[Chapter 4]{dav} for a detailed description of the circle method over the major arcs for the classical case of diagonal equations. In order to apply this to $\int_{\mathfrak{M}(\Delta)}E(\alpha)d\alpha$, we will first fix $\underline{y}$ and thus consider the diagonal equation $f(\underline{x},\underline{y})=f_{\underline{y}}(\underline{x})=0$; afterwards we will take the sum of the obtained expression over all admitted $\underline{y}$.

Since we fix $\underline{y}$, we only look at $x_i$ satisfying $1/|a_iy_{i}^3|^{1/2}\leqslant |x_i|\leqslant (B/|a_iy_{i}^3|)^{1/2}$. Most of the time, it suffices to consider only positive $x_i$: we will denote the corresponding interval for positive $x_i$ with $D_i$, i.e.,
\begin{equation}\label{defD_i}
D_{i}=[1/|a_iy_{i}^3|^{1/2},B^{1/2}/|a_iy_{i}^3|^{1/2}].
\end{equation}
We will also use the notation
\begin{equation}\label{B_{a_i,y_i}}
B_{a_{i},y_{i}}=B^{1/2}/|a_{i}y_{i}^3|^{1/2}.
\end{equation}
Note that, since we consider only $\underline{y}$ with $1\leqslant|a_iy_i^3|\leqslant B$, it holds that $1 \leqslant B_{a_{i},y_{i}}\leqslant B^{1/2}$ for each $i\in\{0,\ldots,n\}$.

Because we first wish to examine the exponential sum $E(\alpha)$ (for $\alpha \in \mathfrak{M}(\Delta)$) for some $\underline{y}$ fixed, we denote this part of $E(\alpha)$ by
$$
E_{\underline{y}}(\alpha)=\mathop{\sum_{1/|a_{i}y_{i}^3|^{1/2}\leqslant |x_{i}|\leqslant B_{a_{i},y_{i}}}}_{i=0,\ldots,n} e(\alpha f_{\underline{y}}(\underline{x})).
$$
Furthermore, for every positive integer $q$ and every integer $a$ relatively prime to $q$, we define
\begin{equation}\label{def:sigma}
\sigma_{\underline{y}}\left(\frac{a}{q}\right)=q^{-(n+1)}\sum_{\underline{z} \in (\mathbf{Z}/q\mathbf{Z})^{n+1}} e((af_{\underline{y}}(\underline{z}))/q).
\end{equation}  
and for every $\beta \in \R$,
\begin{equation}
\tau_{\underline{y},B}(\beta)=\int_{D_{0}}\cdots \int_{D_{n}} e(\beta f_{\underline{y}}(\underline{x}))d\underline{x}.
\end{equation}
\begin{prop}\label{prop:E_y} For $\alpha=a/q+\beta \in \mathfrak{M}(\Delta;q,a)$, we have 
$$ 
E_{\underline{y}}(\alpha)=2^{n+1}\sigma_{\underline{y}}\left(\frac{a}{q}\right)\tau_{\underline{y},B}(\beta)+O\left(q\frac{\sum_{i=0}^{n}|a_{i}y_{i}^3|^{1/2}}{\prod_{i=0}^{n}|a_{i}y_{i}^3|^{1/2}}B^{(n+2)/2}P^{\Delta-2}\right)
$$
under the condition $BP^{\Delta-2}\geqslant1$ on $P$ and $\Delta$. 
\end{prop}
\begin{proof}
Combining positive and negative signs of $x_i$, we have 
\begin{equation} \label{vaste_y}
E_{\underline{y}}(\alpha)=2^{n+1}e(-\alpha t)\prod_{i=0}^n \sum_{x_{i} \in D_{i}}e(\alpha a_{i}x_{i}^2y_{i}^3).
\end{equation}
For $\alpha = a/q+\beta$, the inner sum over $x_i$ equals 
\begin{equation} \label{innersum}
\sum_{1\leqslant z_{i}\leqslant q} e\left(\frac{a a_{i}z_{i}^2y_{i}^3}{q}\right) \mathop{\sum_{v_{i} \in \mathbf{Z}}}_{qv_{i}+z_{i} \in D_{i}}e(\beta a_{i}(q v_{i}+z_{i})^2y_{i}^3).
\end{equation}
Euler's summation formula (in its simplest version) implies  
\begin{equation*}
\sum_{X\leqslant qv+z\leqslant Y}e(\zeta(qv+z)^2)=\frac{1}{q}\int_{X}^Ye(\zeta \eta^2)d\eta + O\left(1+\frac{Y}{q}|\zeta|qY\right)
\end{equation*}
for any real numbers $0\leqslant X<Y,\ \zeta \in \R,\ q,z \in \mathbf{N}$. Taking $Y=B_{a_i,y_i}$ and $\zeta=\beta a_iy_i^3$ and recalling the definition of $D_i$ in \eqref{defD_i}, we can rewrite \eqref{innersum} as 
\begin{equation*}
\sum_{1\leqslant z_i\leqslant q} e\left(\frac{a a_iz_{i}^2y_{i}^3}{q}\right) \left( \frac{1}{q} \int_{D_i}e(\beta a_ix_i^2y_i^3)dx_i+O(1+|\beta|B)\right).
\end{equation*}
We substitute these expressions successively back into \eqref{vaste_y}, and obtain the desired main term. Using the trivial upper bounds
\begin{equation*}
\left|\sum_{x_i\in D_i}e(\alpha a_ix_i^2y_i^3)\right|+\left|\frac{1}{q}\sum_{1\leqslant z_i\leqslant q} e\left(\frac{a a_iz_{i}^2y_{i}^3}{q}\right)\int_{D_i}e(\beta a_ix_i^2y_i^3)dx_i\right| \ll B_{a_i,y_i},
\end{equation*}
we get the total error term $O\left(q(1+|\beta|B)\max_{0\leqslant i\leqslant n}\prod_{j\neq i}B_{a_j,y_j}\right)$. Using \eqref{B_{a_i,y_i}} and $1+|\beta|B\ll P^{\Delta-2}B$, we complete the proof. 
\end{proof} 
From this result, we can now derive an expression for the integral of $E_{\underline{y}}(\alpha)$ over $\mathfrak{M}(\Delta)$ by first integrating the obtained expression of $E_{\underline{y}}(\alpha)$ in Proposition \ref{prop:E_y} over $\mathfrak{M}(\Delta;q,a)$ and then summing over all admitted $a$ and $q$. 

We first define
\begin{equation*}
\mathfrak{I}_{\underline{\varepsilon},t,B}(L)=\int_{|\gamma|<L}e(-\gamma t/B)d\gamma\int_{[B^{-1/2},1]^{n+1}}e(\gamma\sum_{i=0}^{n}\varepsilon_i x_{i}'^{2})d\underline{x}',
\end{equation*}
(where $\varepsilon_i=\sgn(a_{i}y_{i})$) and  
\begin{equation*}
\mathfrak{S}_{\underline{y},\underline{a},t}(L)=\sum_{q\leqslant L}\sum_{\stackrel{0<\frac{a}{q}\leqslant 1}{\gcd(a,q)=1}}\sigma_{\underline{y}}\left(\frac{a}{q}\right).
\end{equation*}
We have 
\begin{equation*}
\int_{|\beta|<P^{\Delta-2}}\tau_{\underline{y},B}(\beta)d\beta=\frac{B^{(n-1)/2}}{\prod_{i=0}^n|a_iy_i^3|^{1/2}}\mathfrak{I}_{\underline{\varepsilon},t,B}(BP^{\Delta-2})
\end{equation*}
and therefore
\begin{multline}\label{eq:overy}
\int_{\mathfrak{M}(\Delta)}E_{\underline{y}}(\alpha)d\alpha=\\ \frac{2^{n+1}\mathfrak{S}_{\underline{y},\underline{a},t}(P^{\Delta})\mathfrak{I}_{\underline{\varepsilon},t,B}(BP^{\Delta-2})}{\prod_{i=0}^{n}|a_{i}y_{i}^3|^{1/2}}\cdot B^{(n-1)/2}\ + O\left(\frac{\sum_{i=0}^{n}|a_{i}y_{i}^3|^{1/2}}{\prod_{i=0}^n|a_{i}y_{i}^3|^{1/2}}B^{(n+2)/2}P^{5\Delta-4}\right).
\end{multline}
Note that the integral $\mathfrak{I}_{\underline{\varepsilon},t,B}(L)$ only depends on the sign of $\underline{y}$ and $\underline{a}$ and no longer on its actual values. 

Next, we make the coefficient of $B^{(n-1)/2}$ in this expression independent of $B$. We first focus on the factor $\mathfrak{S}_{\underline{y},\underline{a},t}(P^{\Delta})$.
\subsubsection{The singular series}
\begin{lem}\label{lem:singular_series}
We have 
$$
\left|\sigma_{\underline{y}}\left(\frac{a}{q}\right)\right|\ll q^{-(n+1)/2}\cdot \prod_{i=0}^{n}\gcd(a_{i}y_{i}^3,q)^{1/2}.
$$
\end{lem}
\begin{proof}
Using elementary properties of generalized Gauss sums (see for example \cite[Chapter 1]{gauss}), we obtain for positive integers $a$ and $c$ that
\begin{equation*}
\left|\sum_{n=0}^{c-1}e\left(\frac{an^2}{c}\right)\right|\ll \gcd(a,c)^{1/2}\sqrt{c}. 
\end{equation*}
Applying this to \eqref{def:sigma} implies the statement. 
\end{proof}
\begin{cor}
For $n\geqslant 4$, it holds that 
\begin{equation}\label{singular_series}
\mathfrak{S}_{\underline{y},\underline{a},t}=\sum_{q=1}^{\infty}\sum_{\stackrel{0<\frac{a}{q}\leqslant 1}{\gcd(a,q)=1}}\sigma_{\underline{y}}\left(\frac{a}{q}\right),
\end{equation}
called \textit{the singular series}, converges absolutely, and
\begin{equation}
\mathfrak{S}_{\underline{y},\underline{a},t}(P^{\Delta})=\mathfrak{S}_{\underline{y},\underline{a},t}+O\left(\frac{\prod_{i=0}^{n}|a_{i}y_{i}^3|^{1/2+\varepsilon}}{\lcm(a_0y_0^3,\ldots,a_ny_n^3)^{2}}\cdot P^{\Delta(-n+3)/2}\right).
\label{eq:singularseries}
\end{equation}
\end{cor}
\begin{proof}
From the previous lemma, we deduce that 
\begin{align*}
\mathfrak{S}_{\underline{y},\underline{a},t}&\ll \sum_{q=1}^{\infty}q^{-(n-1)/2}\prod_{i=0}^{n}\gcd(a_iy_i^3,q)^{1/2}\\
&\ll \sum_{\substack{d_i|a_iy_i^3 \\i=0,\ldots,n}}\frac{(d_0\cdots d_n)^{1/2}}{\lcm(d_0,\ldots,d_n)^{(n-1)/2}}\sum_{q=1}^{\infty}q^{-(n-1)/2}.
\end{align*}
Since $n\geqslant 4$, the latter expression converges and we get
\begin{align*}
\mathfrak{S}_{\underline{y},\underline{a},t}&\ll \sum_{\substack{d_i|a_iy_i^3 \\i=0,\ldots,n}}\frac{(d_0\cdots d_n)^{1/2}}{\lcm(a_0y_0^3,\ldots,a_0y_0^3)^{3/2}}\\
&\ll \frac{\prod_{i=0}^n |a_iy_i^3|^{1/2+\varepsilon}}{\lcm(a_0y_0^3,\ldots,a_0y_0^3)^{3/2}}
\end{align*} 
for any $\varepsilon>0$. 
Moreover, we obtain in the same way that
\begin{align*}
\left|\mathfrak{S}_{\underline{y},\underline{a},t}-\mathfrak{S}_{\underline{y},\underline{a},t}(P^{\Delta}) \right|& \leqslant \sum_{q>P^{\Delta}} q^{-(n-1)/2}\prod_{i=0}^{n}\gcd(a_iy_i^3,q)^{1/2}\\
&\ll \sum_{\substack{d_i|a_iy_i^3 \\0\leqslant i\leqslant n}}\frac{(d_0\cdots d_n)^{1/2}}{\lcm(d_0,\ldots,d_n)^{3/2}}\sum_{q>P^{\Delta}/\lcm(d_0,\ldots,d_n)}^{\infty}q^{-(n-1)/2}\\
&\ll\frac{\prod_{i=0}^{n}|a_iy_i^3|^{1/2+\varepsilon}}{\lcm(a_0y_0^3,\ldots,a_ny_n^3)^2}\cdot P^{\Delta(-n+3)/2}.
\end{align*}
\end{proof}
\begin{remark}\label{convergentiesingularseries}
One can prove (see e.g. \cite[Lemma 5.2 - 5.3]{dav}) for $n \geqslant 4$ that $\mathfrak{S}_{\underline{y},\underline{a},t}$ can be written as an Euler product of $p$-adic densities
$$
\lim_{l\rightarrow \infty} \frac{ \# \{(x_{0},\ldots,x_{n}) \in (\mathbf{Z}/p^l\mathbf{Z})^{n+1}:\sum_{i=0}^{n}a_{i}y_{i}^3x_{i}^2 \equiv t\bmod p^l \}} {p^{ln}}.
$$
\end{remark}
\subsubsection{The singular integral}
Examining $\mathfrak{I}_{\underline{\varepsilon},t,B}(BP^{\Delta-2})$ in (\ref{eq:overy}), we have the following proposition.
\begin{prop}\label{prop:integraal}
For $n\geqslant 3$, we have
\begin{equation}
\mathfrak{I}_{\underline{\varepsilon},t,B}(BP^{\Delta-2})=\mathfrak{I}_{\underline{\varepsilon},t,B}+O\left(B^{(1-n)/2}P^{(\Delta-2)(1-n)/2}\right)
\label{eq:singularintegral}
\end{equation}
with 
$$
\mathfrak{I}_{\underline{\varepsilon},t,B}=\int_{-\infty}^{+\infty}e(-\gamma t/B)d\gamma\int_{[B^{-1/2},1]^{n+1}}e(\gamma\sum_{i=0}^{n}\varepsilon_{i} x_{i}^2)d\underline{x}.
$$
under the condition $BP^{\Delta-2}\geqslant 1$.
\end{prop}
\begin{proof}
As proved in \cite[Proof of Theorem 4.1]{dav}, it holds that  
$$
\left| \int_{B^{-1/2}}^{1}e(\gamma\varepsilon_{i} x_{i}^2)dx_{i} \right| \ll \min\{1,|\gamma|^{-1/2}\},
$$
and thus
\begin{equation}\label{eq:upper_bound_int}
\left|\int_{[B^{-1/2},1]^{n+1}}e(\gamma \sum_{i=0}^{n}\varepsilon_{i} x_{i}^2)d\underline{x}\right|\ll \min\{1,|\gamma|^{-1/2}\}^{n+1}.
\end{equation}
This implies that the integral $\mathfrak{I}_{\underline{\varepsilon},t,B}$ converges, since
\begin{equation*}
\left|\mathfrak{I}_{\underline{\varepsilon},t,B}\right|\ll\int_{-\infty}^{+\infty}\min\{1,|\gamma|^{-1/2}\}^{n+1}d\gamma<+\infty.
\end{equation*}
Also,
\begin{equation*}
|\mathfrak{I}_{\underline{\varepsilon},t,B}(BP^{\Delta-2})-\mathfrak{I}_{\underline{\varepsilon},t,B}| 
\ll \int_{|\gamma|>BP^{\Delta-2}}|\gamma|^{-(n+1)/2}d\gamma 
 \ll  B^{(1-n)/2}P^{(\Delta-2)(1-n)/2}.
\end{equation*}
\end{proof}
Defining \textit{the singular integral} as
\begin{equation}
\mathfrak{I}_{\underline{\varepsilon}}=\int_{-\infty}^{+\infty}d\gamma\int_{[0,1]^{n+1}}e(\gamma\sum_{i=0}^{n}\varepsilon_{i} x_{i}^2)d\underline{x},
\label{eq:singular_integral}
\end{equation}
it follows from the last proof that this integral is also convergent. 
\begin{lem}\label{singuliere_integraal}
It holds that $\mathfrak{I}_{\underline{\varepsilon},t,B}\rightarrow \mathfrak{I}_{\underline{\varepsilon}}$ as $B$ goes to infinity. 
\end{lem}
\begin{proof}We have
\begin{align*}
&|\mathfrak{I}_{\underline{\varepsilon},t,B}-\mathfrak{I}_{\underline{\varepsilon}}|\\
&\leqslant \int_{-\infty}^{+\infty}\hspace{-2ex}\left|(e(-\gamma t/B)-1)\right|d\gamma \left|\int_{[B^{-1/2},1]^{n+1}}\hspace{-2ex}e(\gamma\sum_{i=0}^{n}\varepsilon_{i}x_{i}^2)d\underline{x}\right|+\int_{-\infty}^{+\infty}\hspace{-2ex}d\gamma\left| \int_{\left([B^{-1/2},1]^{n+1}\right)^c}\hspace{-2ex}e(\gamma\sum_{i=0}^{n}\varepsilon_{i}x_{i}^2)d\underline{x}\right|\\
&=I_{1}(B,t)+I_{2}(B),
\end{align*}
where $\left([B^{-1/2},1]^{n+1}\right)^c$ denotes the complement of $[B^{-1/2},1]^{n+1}$ in the hypercube $[0,1]^{n+1}$.

For $I_{1}(B,t)$, we obtain, since $|(e(-\gamma t/B)-1)|=2|\sin(\pi \gamma t B^{-1})|\leqslant \min\{2,2\pi|\gamma||t|B^{-1}\}$ and recalling \eqref{eq:upper_bound_int},
$$
I_{1}(B,t) \ll \int_{-\infty}^{+\infty}\min\{1,\pi|\gamma||t|B^{-1}\}\cdot \min\{1,|\gamma|^{-1/2}\}^{n+1}d\gamma.
$$
Splitting up the latter integral into three parts according to the appropriate range of $\gamma$, we get $I_{1}(B,t)\ll|t|B^{-1}$ for $B$ big enough.

For $I_{2}(B)$, one has that
$$
\left|\int_{0}^1e(\gamma\varepsilon_ix_i^2)dx_i\right| \ll \min\{1,|\gamma|^{-1/2}\} \quad \text{ and }\quad 
\left|\int_{0}^{B^{-1/2}}e(\gamma \varepsilon_ix_i^2)dx_i\right|\ll \min\{B^{-1/2},|\gamma|^{-1/2}\}.
$$
Applying the exclusion-inclusion principle to $I_2(B)$ and observing the symmetric form of the integrand, we get
\begin{equation*}
I_2(B)\ll \sum_{i=1}^{n+1}\int_{-\infty}^{+\infty}\min\{B^{-1/2},|\gamma|^{-1/2}\}^i\cdot \min\{1,|\gamma|^{-1/2}\}^{n+1-i}d\gamma.
\end{equation*}
It follows that  $I_2(B)\ll B^{-1/2}$. Hence,
\begin{equation}\label{eq:sing_int}
|\mathfrak{I}_{\underline{\varepsilon},t,B}-\mathfrak{I}_{\underline{\varepsilon}}|\ll_t B^{-1/2}
\end{equation}
for $B$ big enough, completing the proof. 

\end{proof}
Note that from Proposition \ref{prop:integraal} and \eqref{eq:sing_int}, one has
\begin{equation}\label{eq:int_sing}
\mathfrak{I}_{\underline{\varepsilon},t,B}(BP^{\Delta-2})=\mathfrak{I}_{\underline{\varepsilon}}+O\left(B^{-1/2}+B^{(1-n)/2}P^{(\Delta-2)(1-n)/2}\right).
\end{equation}
We now return to the integral of $E_{\underline{y}}(\alpha)$ over the major arcs. 
\begin{prop}\label{prop:major}For $n\geqslant 4$ and for any $\Delta$ with $0<\Delta< 1/5$, there exists a $\delta>0$ so that
\begin{equation}{\label{equ:vgl}}
\int_{\mathfrak{M}(\Delta)}E_{\underline{y}}(\alpha)d\alpha=\frac{2^{n+1}\mathfrak{S}_{\underline{y},\underline{a},t}\mathfrak{I}_{\underline{\varepsilon}}}{\prod_{i=0}^{n}|a_{i}y_{i}^3|^{1/2}}\cdot B^{(n-1)/2}+O_{\underline{y},\underline{a}}\left(B^{(n-1)/2-\delta}\right).
\end{equation}
\end{prop}
\begin{proof}
Substituting \eqref{eq:singularseries} and \eqref{eq:int_sing} into formula \eqref{eq:overy} we obtained for $\int_{\mathfrak{M}(\Delta)}E_{\underline{y}}(\alpha)d\alpha$, we get 
\begin{multline}\label{major_voorlopig}
\int_{\mathfrak{M}(\Delta)}E_{\underline{y}}(\alpha)d\alpha=\frac{2^{n+1}\mathfrak{S}_{\underline{y},\underline{a},t}\mathfrak{I}_{\underline{\varepsilon}}}{\prod_{i=0}^{n}|a_{i}y_{i}^3|^{1/2}}\cdot B^{(n-1)/2}+O\Biggl(\frac{\prod_{i=0}^n |a_iy_i^3|^{\varepsilon}}{\lcm(a_0y_0^3,\ldots,a_ny_n^3)^{2}}\cdot B^{(n-1)/2}P^{\Delta(-n+3)/2}+\\  \frac{B^{(n-2)/2}+P^{(\Delta-2)(1-n)/2}}{\prod_{i=0}^n|a_{i}y_{i}^3|^{1/2}}+ \frac{\sum_{i=0}^{n}|a_{i}y_{i}^3|^{1/2}}{\prod_{i=0}^n|a_{i}y_{i}^3|^{1/2}}\cdot B^{(n+2)/2}P^{5\Delta-4}\Biggr). 
\end{multline}

For this expression to be nontrivial, we have to determine $P=P(B)$ and $\Delta$ properly (under the condition $BP^{\Delta-2}\geqslant 1$) that the error term is $O_{\underline{y},\underline{a}} (B^{(n-1)/2-\delta})$ for some $\delta>0$. Taking $P=B^{1/2}$ and $0<\Delta<1/5$ is satisfactory.  
\end{proof}
We can now prove our estimate for the major arcs. 
\begin{theo} \label{theorem_major} For $n\geqslant 4$ and for any $\Delta$ with $0<\Delta<1/15$, there exists a $\delta>0$ so that
$$
\int_{\mathfrak{M}(\Delta)}E(\alpha)d\alpha=C_{\underline{a},t}\cdot B^{(n-1)/2}+O\left(B^{(n-1)/2-\delta}\right)
$$ 
with
$$
C_{\underline{a},t}=2^{n+1}\sum_{\underline{y}\in \Z_0^{n+1}}\left(\prod_{i=0}^{n}\mu^2(|y_{i}|)\right)\frac{\mathfrak{S}_{\underline{y},\underline{a},t}\mathfrak{I}_{\underline{\varepsilon}}}{\prod_{i=0}^{n}|a_{i}y_{i}^3|^{1/2}}
$$
with $\mathfrak{S}_{\underline{y},\underline{a},t}$ and $\mathfrak{I}_{\underline{\varepsilon}}$ as defined above. 
\end{theo}
\begin{proof}
We sum \eqref{major_voorlopig} over all squarefree $y_i$ such that $1\leqslant |a_iy_i^3|\leqslant B$, $i\in \{0,\ldots,n\}$ and denote the sum of the coefficient of the main term by $C_{\underline{a},t}(B)$. 

We obtain, using lemma \ref{lem:singular_series},
\begin{equation}\label{eq:upper}
\frac{\mathfrak{S}_{\underline{y},\underline{a},t}}{{\prod_{i=0}^{n}|a_iy_i^3|^{1/2}}}\ll \frac{\prod_{i=0}^n|a_iy_i^3|^{\varepsilon}}{\lcm(a_0y_0^3,\ldots,a_ny_n^3)^2},
\end{equation}
for any $\varepsilon>0$.

We have
\begin{align}\label{tau-functie}
 \sum_{\min\limits_{0\leqslant i \leqslant n}|a_iy_i^3|\geqslant B}\frac{\prod_{i=0}^n|a_iy_i^3|^{\varepsilon}}{\lcm(a_0y_0^3,\ldots,a_ny_n^3)^2}&\ll B^{\varepsilon}\sum_{n\geqslant B}\frac{\#\{(y_0,\ldots,y_n): \lcm(a_0y_0^3,\ldots,a_0y_n^3)=n\}}{n^2}\nonumber\\ 
&\ll B^{\varepsilon}\sum_{n\geqslant B}\frac{\tau(n)^{n+1}}{n^2}\ll B^{-1+\varepsilon},
\end{align}
for any $\varepsilon>0$. This allows us to replace $C_{\underline{a},t}(B)$ by $C_{\underline{a},t}$.

We now turn to the error term in \eqref{major_voorlopig} summing over all admitted values of $\underline{y}$ and putting $P=B^{1/2}$ as before.

The first error term can be treated as the main term. The coefficient of the third and fourth error term will also converge without any extra condition. Moreover, the upper bound can be made independent of the $a_i$. For the last error term however, the coefficient will asymptotically contribute $O(B^{1/3})$.

This means the extra condition 
$$
\frac{1}{3}+\frac{n+2}{2}+\frac{5\Delta-4}{2}<\frac{n-1}{2} \Leftrightarrow \Delta<\frac{1}{15}
$$
has to be satisfied for the error term to behave properly. This proves the statement.
\end{proof}
Note that \eqref{eq:upper} and \eqref{tau-functie} also provides a uniform upper bound of $C_{\underline{a},t}$, i.e., $C_{\underline{a},t}\leqslant C$, independently of $\underline{a}$ and $t$.

%% file: towards_the_main_problem_revision_ANT_June_2011.tex
\section{Towards the main problem}\label{section:mobius}
Combining the previous results, we are able to prove the following theorem.
\begin{theo}\label{M_{a,t}(B)}
For $n\geqslant 4$, there exists a $\delta>0$ so that
$$
\#M_{\underline{a},t}(B)=C_{\underline{a},t}\cdot B^{(n-1)/2}+O\left(B^{(n-1)/2-\delta}\right),
$$
with the constant $C_{\underline{a},t}$ described in Theorem \ref{theorem_major}.
\end{theo}
\begin{proof} This follows directly from Theorem \ref{theorem_major} and Theorem \ref{theo:minor} and \eqref{eq:M_{a,t}(B)}.
\end{proof}
\begin{remark}\label{opmerking}
Note that the error term is independent of $\underline{a}$ and $t$ and recall we also proved $C_{\underline{a},t}$ can be bounded uniformly independent of $\underline{a}$ and $t$. This implies that $\#M_{\underline{a},t}(B)\leqslant C\cdot B^{(n-1)/2}$ for some constant $C>0$. Indeed, when $B<1$, $M_{\underline{a},t}(B)=\emptyset$ and for $B\geqslant 1$, it follows from Theorem \ref{M_{a,t}(B)} that $\# M_{\underline{a},t}(B)\leqslant C'\cdot B^{(n-1)/2}+C''\cdot  B^{(n-1)/2-\delta}\leqslant C\cdot B^{(n-1)/2}$, where $C=2\max\{C',C''\}$.
\end{remark}
Going back to $M(B)$ (see Definition \ref{different sets}), we will now prove the following theorem. 
\begin{theo}\label{theo:final} For $n\geq 4$, there exists an explicit constant $D$ and a $\delta>0$ such that
$$ 
\#M(B)=D\cdot B^{(n-1)/2}+O\left(B^{(n-1)/2-\delta}\right)
$$
as $B$ goes to infinity. 
\end{theo} 
(The definition of the constant $D$ is given in Lemma \ref{convergentie_D}; in the next section, we will give some indications about the interpretation of $D$.)

The only problem still left in order to prove Theorem \ref{theo:final}, is to understand how we can tackle the additional gcd condition $\gcd(x_0y_0,\ldots,x_ny_n)=1$ on the solutions. Note that the M\"obius inversion at hand leads to divisibility conditions on both $x_i$ and $y_i$ which have to be handled with care.  

Let $\underline{e}=(e_0,\ldots,e_n)\in \mathbf{N}_0^{n+1}$ and $\underline{f}=(f_0,\ldots,f_n)\in \mathbf{N}_0^{n+1}$ where $f_i$ is squarefree for each $i \in \{0,\ldots,n\}$. 
\begin{definition}
We denote the set
$$
\left\{(\underline{x},\underline{y})\in \Z_{0}^{2n+2} : \sum_{i=0}^{n}x_{i}^2y_{i}^3=0,\ \max_{0\leqslant i\leqslant n}|x_{i}^2y_{i}^3|\leqslant B,\ e_i|x_i,\ f_i|y_i\text{ and }y_i\text{ squarefree for all }i\right\}
$$
by $N_{(\underline{e},\underline{f})}(B)$. 
\end{definition}
Demanding that solutions in $N_{(\underline{1},\underline{1})}(B)(=M_{(1,\ldots,1),0}(B))$ satisfy $\gcd(x_0y_0,\ldots,x_ny_n)=1$ means we wish to leave out those solutions of $N_{(\underline{1},\underline{1})}(B)$ for which there exists a prime $p$ and a subset $I \subset \{0,\ldots,n\}$ such that $p|x_{i}$ if $i \in I$ and $p|y_{i}$ if $i \notin I$ (or $i \in I^{c}$, where $I^{c}$ denotes the complement of $I$ in $\{0,\ldots,n\}$) in order to get to $M(B)$. Defining, for a prime $p$ and subsets $I,J\subset \{0,\ldots,n\}$, the couple $(\underline{e}^{p,I},\underline{f}^{p,J})$ by $e_{i}^{p,I}=p$ for $i \in I$ and $e_{i}^{p,I}=1$ otherwise and analogously for $\underline{f}^{p,J}$, it hence follows that  
\begin{eqnarray}\label{union}
M(B)= N_{(\underline{1},\underline{1})}(B)\ \setminus \bigcup_{(p,I)} N_{(\underline{e}^{p,I},\underline{f}^{p,I^{c}})}(B).
\end{eqnarray}
Notice that in this last union only a finite number of sets are nonempty since for a prime $p\geq \sqrt{B}$, we get $N_{(\underline{e}^{p,I},\underline{f}^{p,I^{c}})}(B)=\emptyset$. 

\begin{definition}\label{definition:mu}
Let $S$ be a finite set of couples $(p,I)$. To $S$, we can associate a couple $(\underline{e},\underline{f})$ as follows: defining for each prime $p$ the index sets $I_p=\cup_{(p,I)\in S}I$ and $J_p=\cup_{(p,I)\in S}I^{c}$, the associated couple is given by $e_i=\prod_{\{p | i\in I_{p}\}}p$ and $f_i=\prod_{\{p | i \in J_{p}\}}p$.

We then define
\begin{equation*}
\mu(\underline{e},\underline{f})=\sum_{n\geq 0}(-1)^n \#\{\text{sets } S\text{ with cardinality }n\text{ such that the associated couple is }(\underline{e},\underline{f})\}.
\end{equation*}
\end{definition}
Observing \eqref{union} together with this definition, we have 
\begin{equation}
\#M(B)=\sum_{e=1}^{\infty}\mathop{\sum_{(\underline{e},\underline{f})\in \mathbf{N}^{2n+2}}}_{e=\gcd(e_if_i,\ i=0,\ldots,n)}\mu(\underline{e},\underline{f})\cdot \#N_{(\underline{e},\underline{f})}(B).
\label{M(B)}
\end{equation}
The following lemma collects some properties of $\mu$. 
\begin{lem}\label{mu-functie}
There exists a function $\widetilde{\mu}:\Z^{2n+2}\to \Z$ such that 
\begin{enumerate}
\item[(i)] $\mu(\underline{e},\underline{f})=\prod_p\widetilde{\mu}(v_p(\underline{e}),v_p(\underline{f}))$, where $v_{p}(\underline{e})=(v_{p}(e_{0}),\ldots,v_{p}(e_{n}))$ (and analogously for $v_{p}(\underline{f})$),
\item[(ii)] $\widetilde{\mu}(\underline{m},\underline{n})=0$ if $m_i=n_i=0$ and $(\underline{m},\underline{n})\neq(\underline{0},\underline{0})$ or if $m_i>1$ for some $i$,
\item[(iii)] $\sum_{I\cup J=\{0,\ldots,n\}}|\widetilde{\mu}(I,J)|\leqslant 2^{2^{n+1}}$, where, for subsets $I,J\subset \{0,\ldots,n\}$, $\widetilde{\mu}(I,J)$ denotes $\widetilde{\mu}(m_{0}^{I},\ldots,m_{n}^{I},{m_{0}'}^{J} ,\ldots,{m_{n}'}^{J})$ with $m_{i}^{I}=1$ if $i \in I$ and $m_{i}^{I}=0$ otherwise and ${{m_{i}}'}^{J}=1$ if $i\in J$ and ${m_{i}'}^{J}=0$ otherwise.
\end{enumerate}
\end{lem}
\begin{proof}
(i) and (ii) follow directly from the definition of $\mu$ in Definition \ref{definition:mu}. Moreover, from this definition, it follows, if $I \cup J =\{0,\ldots,n\}$ and denoting by $T$ a finite set of subsets $I \subset \{0,\ldots,n\}$, that
$$
\widetilde{\mu}(I,J)=\sum_{m}(-1)^{m}\#\{\textrm{sets }T\text{ of cardinality }m\text{ such that }I=\cup_{K \in T} K\text{ and }J=\cup_{K \in T} K^c\}. 
$$ 
If we sum over all possible $I$ and $J$ such that $I \cup J = \{0,\ldots,n\}$, we get (iii).
\end{proof}
Consider now $N_{(\underline{e},\underline{f})}(B)$ for a couple $(\underline{e},\underline{f})$ for which $\mu(\underline{e},\underline{f})\neq 0$ and $\gcd(e_{i}f_{i},\ i=0,\ldots,n)=e$, i.e. a subset with nontrivial contribution to $\#M(B)$ (recall (\ref{M(B)})). Since $\#N_{(\underline{e},\underline{f})}(B)=\#M_{\underline{e^2f^3},0}(B)$ (where $\underline{e^2f^3}=(e_0^2f_0^3,\ldots,e_n^2f_n^3)$), we know by Theorem \ref{M_{a,t}(B)} that $\#N_{(\underline{e},\underline{f})}(B)= C_{\underline{e^2f^3},0}\cdot B^{(n-1)/2} + O(B^{(n-1)/2-\delta})$. Since $e$ divides $e_if_i$, we can write $e_i^2f_i^3=v_ie^2$ for some $v_{i} \in \mathbf{N}$ for each $i\in \{0,\ldots,n\}$. Making the substitution $x_i'=x_i/e_i$ and $y_i'=y_i/f_i$, we see that $N_{(\underline{e},\underline{f})}(B)$ corresponds to the set
$$
\left\{(\underline{x'},\underline{y'})\in \Z_0^{2n+2}: \sum_{i=0}^{n}v_i{x_{i}'}^{2}{y_{i}'}^{3}=0,\ \max_{0\leqslant i \leqslant n}|v_i{x_{i}'}^{2}{y_{i}'}^{3}|\leqslant \frac{B}{e^2}\text{ and }y_i'\text{ squarefree}\right\} 
$$
where we eliminated $e^2$ in the equation and hence, $\#N_{(\underline{e},\underline{f})}(B)=\#M_{\underline{v},0}\left(B/e^2\right)$. Letting $B$ go to infinity, this implies that the main terms in the asymptotic formulas of $\#N_{\underline{e},\underline{f}}(B)$ and $\#M_{\underline{v},0}\left(B/e^2\right)$ are equal, and in particular, that 
\begin{equation}\label{gelijkheid}
\#N_{(\underline{e},\underline{f})}(B)-C_{\underline{e^2f^3},0}\cdot B^{(n-1)/2}=O\left( \frac{B^{(n-1)/2-\delta}}{e^{n-1-2\delta}}\right).
\end{equation}
Notice we also obtain (recall Remark \ref{opmerking}) that
\begin{equation}\label{equ:n1}
\#N_{(\underline{e},\underline{f})}(B) \leqslant C\cdot\frac{B^{(n-1)/2}}{e^{n-1}} \quad \text{and}\quad  C_{\underline{e^2f^3},0}\leqslant\frac{C}{e^{n-1}}.
\end{equation}
From these results, we can now prove 
\begin{lem}\label{convergentie_D} The series 
$$
D=\sum_{e=1}^{\infty}\mathop{\sum_{(\underline{e},\underline{f})\in \mathbf{N}^{2n+2}}}_{\gcd(e_{i}f_{i},\ i=0,\ldots,n)=e} \mu(\underline{e},\underline{f})\cdot C_{\underline{e^2f^3},0}
$$
converges.
\end{lem}
\begin{proof}
Substituting \eqref{equ:n1} into the definition of $D$ and using the properties of $\mu$ in Lemma \ref{mu-functie}, we get 
\begin{align*}
|D| &\ll  \sum_{e=1}^{\infty}\mathop{\sum_{(\underline{e},\underline{f})\in \mathbf{N}^{2n+2}}}_{\gcd(e_{i}f_{i},\ i=0,\ldots,n)=e}\frac{|\mu(\underline{e},\underline{f})|}{e^{n-1}}\\
&\leqslant \prod_{p}\sum_{k=0}^{2} \mathop{\sum_{(v_{p}(\underline{e}),v_{p}(\underline{f}))\in \mathbf{N}^{2n+2}}}_{\min_{i}\{v_{p}(e_{i})+v_{p}(f_{i})\}=k} \frac{|\mu_{p}(v_{p}(\underline{e}),v_{p}(\underline{f}))|}{p^{k(n-1)}}\\
&\leqslant \prod_{p}\left(1+2\frac{2^{2^{n+1}}}{p^{n-1}}\right),
\end{align*}
which converges since $n \geq 4$.
\end{proof}
\begin{proof}[Proof of Theorem \ref{theo:final}]
From the definition of $D$ and \eqref{gelijkheid}, it follows that
\begin{equation*}
\left|\#M(B)-D\cdot B^{(n-1)/2}\right| \ll
 \sum_{e=1}^{\infty}\mathop{\sum_{(\underline{e},\underline{f})\in \mathbf{N}^{2n+2}}}_{\gcd(e_{i}f_{i},\ i=0,\ldots,n)=e}\left|\mu(\underline{e},\underline{f})\right|\cdot \frac{B^{(n-1)/2-\delta}}{e^{(n-1)-2\delta}}.
\end{equation*}
Following the same reasoning as in Lemma \ref{convergentie_D}, we then get that 
\begin{equation*}
\left|\#M(B)-D\cdot B^{(n-1)/2}\right|  
 \ll  B^{(n-1)/2-\delta}\cdot \prod_{p}\left(1+2\frac{2^{2^{n+1}}}{p^{n-1-2\delta}}\right),
\end{equation*}
 where the product converges for $\delta>0$ small enough since $n\geqslant 4$. This proves the theorem.
%
%
\end{proof}

%% file: adelic_distribution_revision_version_ANT_June_2011.tex
\section{Rational points on the orbifold $(\mathbf{P}^{n-1},\Delta)$}\label{sec:orbi}
\subsection{The main theorem}
We can now prove our main theorem. 
\begin{theo}\label{orbifold} For $n\geq 4$, there exists a $\delta>0$ such that  
$$
\#(\mathbf{P}^{n-1},\Delta)(\Q)_{\leqslant B}= C\cdot B^{(n-1)/2}+O\left(B^{(n-1)/2-\delta}\right).
$$
Here, 
$$
C=\frac{1}{2^{n+1}}\sum_{e=1}^{\infty}\hspace{-3ex}\mathop{\sum_{(\underline{e},\underline{f})\in \mathbf{N}^{2n+2}}}_{\gcd(e_if_i,\ i= 0,\ldots,n)=e}\hspace{-3ex}\mu(\underline{e},\underline{f})\hspace{-2ex}\mathop{\sum_{\underline{y}\in \mathbf{Z}_{0}^{n+1}/\{\pm 1\}}}_{y_{i}\ \mathrm{squarefree}}\frac{2^{n+1}\mathfrak{S}_{\underline{y},\underline{e^2f^3},0}\mathfrak{I}_{\underline{\varepsilon}}}{\prod_{i=0}^n(e_{i}^2f_{i}^{3}|y_{i}^3|)^{1/2}},
$$
with $ \mathfrak{S}_{\underline{y},\underline{a},t}$, $\mathfrak{I}_{\underline{\varepsilon}}$ and the function $\mu$ as defined before. (By $\underline{y}\in \mathbf{Z}_{0}^{n+1}/\{\pm 1\}$, we denote the $n+1$-tuples $(y_0,\ldots,y_n)\in \Z_0^{n+1}$, defined up to sign as $n+1$-tuple.)
\end{theo}
\begin{proof} The connection between $(\mathbf{P}^{n-1},\Delta)(\Q)_{\leqslant B}^{+}$ and the set $M(B)$ given by \eqref{eq:orbifold} in the introduction, together with Theorem \ref{theo:final}, implies that the theorem holds for $\#(\mathbf{P}^{n-1},\Delta)(\Q)_{\leqslant B}^+$. It remains to prove that, for $n\geqslant 4$, the set of points $(a_0:\ldots:a_n) \in (\mathbf{P}^{n-1},\Delta)(\Q)_{\leqslant B}$ with at least one zero coordinate (whose cardinality is $\ll \#(\mathbf{P}^{n-2},\Delta)(\Q)_{\leqslant B}$), is asymptotically negligible compared to $(\mathbf{P}^{n-1},\Delta)(\Q)_{\leqslant B}^+$. 

We will verify this for $n=4$; by induction, the statement follows for $n>4$.  

As mentioned in Remark \ref{speciaal_geval}, it follows from Lemma \ref{lemma} that
\begin{equation*}
\#(\mathbf{P}^2,\Delta)(\Q)_{\leqslant B}^+ \ll B^{1+\varepsilon}.
\end{equation*}
Combining this with the trivial upper bound $\#(\mathbf{P}^1,\Delta)(\Q)_{\leqslant B}\ll B$, we obtain
$$
\#(\mathbf{P}^2,\Delta)(\Q)_{\leqslant B}\ll B^{1+\varepsilon}<B^{3/2}
$$
for $\varepsilon>0$ sufficiently small. 
%
%
%
%
%
\end{proof}
\subsection{Description of the constant}
An alternative description of $(\mathbf{P}^{n-1},\Delta)(\Q)_{\leqslant B}^{+}$ can be obtained as follows. Consider $\underline{y}\in \mathbf{Z}_0^{n+1}/\{\pm 1\}$ with each $y_i$ squarefree. For such $\underline{y}$, let $Q_{\underline{y}}$ denote the smooth quadric defined by the homogeneous polynomial $F_{\underline{y}}(\underline{X})=y_0^3X_0^2+\ldots+y_n^3X_n^2 \in \mathbf{Z}[X_{0},\ldots,X_{n}]$. Furthermore, define the morphism 
\begin{equation}\label{morphisms}
\begin{array}{ccccc}
\pi_{\underline{y}}&:&Q_{\underline{y}} &\rightarrow & H\\
& &(x_{0}:\ldots:x_{n})&\mapsto& (y_{0}^{3}x_{0}^2:\ldots:y_{n}^3x_{n}^2).\\
\end{array}
\end{equation}
We will consider points $(x_0:\ldots:x_n)\in Q_{\underline{y}}(\Q)$ with $x_i\in \Z$, such that $\prod_{i=0}^n x_i\neq 0$ and $\gcd(x_0y_0,\ldots,x_ny_n)=1$. We denote this subset of $Q_{\underline{y}}(\Q)$ by $Q_{\underline{y}}(\Q)^{+}$. This set is mapped into $(\mathbf{P}^{n-1},\Delta)(\Q)^+$ by $\pi_{\underline{y}}$ and, keeping in mind \eqref{eq:orbifold}, we have 
\begin{equation}
(\mathbf{P}^{n-1},\Delta)(\Q)^+= \mathop{\coprod_{\underline{y}\in \mathbf{Z}_{0}^{n+1}/\{\pm1\}}}_{y_{i}\text{ squarefree}} \pi_{\underline{y}}(Q_{\underline{y}}(\Q)^{+}). 
\label{partitie2}
\end{equation}
This implies
\begin{equation*}
\#(\mathbf{P}^{n-1},\Delta)(\Q)_{\leqslant B}^+ =\frac{1}{2^{n+1}} \mathop{\sum_{\underline{y}\in \Z_0^{n+1}/\{\pm 1\}}}_{y_i \text{ squarefree}}\#\left\{(x_0:\ldots: x_n)\in Q_{\underline{y}}(\Q)^+: \max_{0\leqslant i\leqslant n}|x_i^2y_i^3|\leqslant B\right\}
\end{equation*}
For a fixed $\underline{y}$, an asymptotic expression for each of the latter sets is known using the classical circle method (see \cite[Chapter 8]{dav}) and a M\"obius inversion for the gcd condition $\gcd(x_0y_0,\ldots,x_ny_n)=$1. 

Moreover, from Lemma \ref{convergentie_D}, it follows that we can change the order of summation for $e$ and $\underline{y}$ in the constant $C$ from Theorem \ref{orbifold} and thus, defining
\begin{equation}\label{eq:constanteQ_y}
C_{Q_{\underline{y}}}=\sum_{e=1}^{\infty}\mathop{\mathop{\sum_{(\underline{e},\underline{f})\in \mathbf{N}^{2n+2}}}_{\gcd(e_if_i,\ i=0,\ldots,n)=e}}_{f_i|y_i}\mu(\underline{e},\underline{f})\frac{2^{n+1}\mathfrak{S}_{\underline{y},\underline{e}^2,0}\mathfrak{J}_{\underline{\varepsilon}}}{\prod_{i=0}^{n}(e_{i}|y_{i}|^{3/2})},
\end{equation}
we have, for $n\geqslant 4$,
\begin{equation*}
\#(\mathbf{P}^{n-1},\Delta)(\Q)_{\leqslant B} \sim \Biggl( \frac{1}{2^{n+1}}\mathop{\sum_{\underline{y}\in \mathbf{Z}_{0}^{n+1}/\{\pm 1\}}}_{y_{i} \text{ squarefree}}C_{Q_{\underline{y}}}\Biggr)\cdot B^{(n-1)/2}
\end{equation*}
as $B$ goes to infinity.

This constant $C_{Q_{\underline{y}}}$ can be given a more geometrical interpretation using 
the adelic space $Q_{\underline{y}}(\mathbb{A}_{\Q})$ of the quadric $Q_{\underline{y}}$, as explained in \cite[\S 5]{Peyre}. Here, it has been shown that the refined version of the Manin conjecture is compatible with the circle method for smooth quadrics in $\mathbf{P}^{n}_{\Q}$ and moreover, that rational points on smooth quadrics are equidistributed.
Considering the Tamagawa measure $\text{\boldmath{$\omega$}}_{H_{\underline{y}}}$ (corresponding to the height function $H_{\underline{y}}$ defined as $H_{\underline{y}}(P)=\max_{0\leqslant i\leqslant n}|x_{i}^2y_{i}^3|$ where $P=(x_{0}:\ldots:x_{n}) \in Q_{\underline{y}}(\Q)$) on $\Q_{\underline{y}}(\mathbb{A}_{\Q})$, the equidistribution of the rational points on $Q_{\underline{y}}$ implies that for every good open subset $W$ (i.e. an open subset $W$ for which $\text{\boldmath{$\omega$}}_{H_{\underline{y}}}(\partial W)=0$, where $\partial W=\overline{W}\setminus W$) of $Q_{\underline{y}}(\mathbb{A}_{\Q})$, it holds that
$$
\frac{\#\left\{P\in Q_{\underline{y}}(\Q)^+\cap W\mid H_{\underline{y}}(P)\leqslant B\right\}}{\#\left\{P\in Q_{\underline{y}}(\Q)^+\mid H_{\underline{y}}(P)\leqslant B\right\}} \rightarrow \frac{\text{\boldmath{$\omega$}}_{H_{\underline{y}}}(W)} {\text{\boldmath{$\omega$}}_{H_{\underline{y}}}(Q_{\underline{y}}(\mathbb{A}_{\Q}))}
$$
as $B$ goes to infinity. We refer to \cite{Peyre} for more details on this matter.
 This implies we can obtain a description of the constant $C_{Q_{\underline{y}}}$ in terms of the measure $\text{\boldmath{$\omega$}}_{H_{\underline{y}}}$ of a certain subset of the adelic space $Q_{\underline{y}}(\mathbb{A}_{\Q})$ of the quadric $Q_{\underline{y}}$. More precisely, it follows that
$$
C_{Q_{\underline{y}}}=\text{\boldmath{$\omega$}}_{H_{\underline{y}}}(Q_{\underline{y}}(\mathbb{A}_{\Q})^{\dag})/(n-1),
$$
where $Q_{\underline{y}}(\mathbb{A}_{\Q})^{\dag}$ denotes the good open subset of $Q_{\underline{y}}(\mathbb{A}_{\Q})$, defined by the gcd condition $\gcd(x_0y_0,\ldots,x_ny_n)=1$ we imposed on $Q_{\underline{y}}(\Q)$. (Note that imposing the open condition $\prod_{i=0}^n x_i\neq 0$ does not change the measure.)
We obtain the following corollary.
\begin{cor}
For $n\geqslant 4$, we have 
\begin{equation}
\#(\mathbf{P}^{n-1}, \Delta)(\Q)_{\leqslant B}\sim \Biggl(\frac{1}{2^{n+1}}\mathop{\sum_{\underline{y}\in \mathbf{Z}_{0}^{n+1}/\{\pm 1\}}}_{y_{i}\text{ squarefree} }C_{Q_{\underline{y}}}\Biggr)\cdot B^{(n-1)/2}
\label{eq:orbi}
\end{equation}
as $B$ goes to infinity, where $C_{Q_{\underline{y}}}=\text{\boldmath{$\omega$}}_{H_{\underline{y}}}(Q_{\underline{y}}(\mathbb{A}_{\Q})^\dag)/(n-1)$. 
\end{cor}
\subsection{The adelic space of the orbifold $(\mathbf{P}^{n-1},\Delta)$}
In order to define the adelic space of the orbifold properly, we first have to explain how we can translate the definition of \lq squarefulness\rq\ to the different completions of $\Q$. 

At each finite place $v=p$, a $p$-adic integer $a \in \Z_{p}$ is squareful if $v_{p}(a)\neq 1$. Due to the structure of $\Q_p^{\times}$, this means that we can write a squareful $p$-adic integer $a$ uniquely as $x^2y^3$, with $x\in \Z_p^\times$ and $y \in \Z$ squarefree. 

On the other hand, any real number $a\in \R$ can we written as $(\pm1)^3x^2$ and ought to be considered as squareful.

Since we identified $(\mathbf{P}^{n-1},\Delta)(\Q)$ with $\left\{(u_0:\ldots:u_n)\in H(\Q)\ |\ u_i \text{ is squareful for each }i\right\}$ (recall $H\subset \mathbf{P}^{n}$ is the hyperplane defined by $X_0+\cdots+X_n=0$), we have, for each $v\in \Val(\Q)$, that
\begin{align*}
(\mathbf{P}^{n-1},\Delta)(\Q_v)&=\{(u_0:\ldots:u_n)\in H(\Q_v): u_i\text{ is squareful for each } i\}\\
&=\{(x_{0,v}^2y_0^3:\ldots:x_{n,v}^2y_n^3)\in H(\Q_v): \underline{y}\in \Z_0^{n+1}/\{\pm 1\},\ y_i\text{ squarefree}\}.
\end{align*}
This implies, recalling the definition of $\pi_{\underline{y}}$ in \eqref{morphisms},
\begin{equation}
(\mathbf{P}^{n-1},\Delta)(\Q_v)=\mathop{\bigcup_{\underline{y}\in \mathbf{Z}_{0}^{n+1}/\{\pm1\}}}_{y_{i}\text{ squarefree}}\pi_{\underline{y}}(Q_{\underline{y}}(\Q_v)^{\dag}),
\end{equation}
where for a finite place $v=p$, $Q_{\underline{y}}(\Q_p)^{\dag}$ is the open subset of $Q_{\underline{y}}(\Q_p)$ defined by the condition $\min_{0\leqslant i\leqslant n}(v_p(x_{i,p}y_{i}))=0$ and where $Q_{\underline{y}}(\R)^{\dag}=Q_{\underline{y}}(\R)$. 

Note that the considered union is not disjoint, but that the image for different $\underline{y}$ and $\underline{y}'$ either coincides or is disjoint. Hence, it follows that, at each place $v \in \Val(\Q)$, $(\mathbf{P}^{n-1},\Delta)(\Q_{v})$ can be described as a finite disjoint union of sets $\pi_{\underline{y}}(Q_{\underline{y}}(\Q_{v})^{\dag})$ for specified $\underline{y} \in \mathbf{Z}_{0}^{n+1}/\{\pm 1\}$.

\begin{definition}\label{adelic_space_orbifold} We define the adelic space $(\mathbf{P}^{n-1},\Delta)(\mathbb{A}_{\Q})$ as 
$$
(\mathbf{P}^{n-1},\Delta)(\mathbb{A}_{\Q})=\prod_{v\in \Val(\Q)}(\mathbf{P}^{n-1},\Delta)(\Q_{v}).
$$
\end{definition}
\begin{remark}
One may prove that $(\mathbf{P}^{n-1},\Delta)(\Q)$ is dense in $(\mathbf{P}^{n-1},\Delta)(\mathbb{A}_{\Q})$. This follows from the fact that weak approximation holds for smooth quadrics. 
\end{remark}
\subsection{Distribution of rational points on $(\mathbf{P}^{n-1},\Delta)$} We can now consider the probability measure 
\begin{equation}\label{eq:prob_measure}
\mu_{H\leqslant B}^{(\mathbf{P}^{n-1},\Delta)}=\frac{1}{\#(\mathbf{P}^{n-1},\Delta)(\Q)_{\leqslant B}}\mathop{\sum_{P\in (\mathbf{P}^{n-1},\Delta)(\Q)}}_{H(P)\leqslant B}\delta_{P}
\end{equation}
on $(\mathbf{P}^{n-1},\Delta)(\mathbb{A}_{\Q})$. Here, we will investigate the convergence of $\mu_{H\leq B}^{(\mathbf{P}^{n-1},\Delta)}$ to a specific measure on the adelic space of the orbifold which we have yet to define, when $B$ goes to infinity. Keeping in mind the description of $(\mathbf{P}^{n-1},\Delta)(\mathbb{A}_{\Q})$ we gave above, we can define this measure in the following natural way. 
\begin{definition}
We define the measure $\text{\boldmath{$\omega$}}_{(\mathbf{P}^{n-1},\Delta)}$ on $(\mathbf{P}^{n-1},\Delta)(\mathbb{A}_{\Q})$ as
\begin{equation}
\text{\boldmath{$\omega$}}_{(\mathbf{P}^{n-1},\Delta)}(U)=\mathop{\sum_{\underline{y}\in \mathbf{Z}_{0}^{n+1}/\{\pm1\}}}_{y_{i}\text{ squarefree}}\text{\boldmath{$\omega$}}_{H_{\underline{y}}}(\pi_{\underline{y}}^{-1}(U)),
\label{measure}
\end{equation}
where $U$ is an open subset of $(\mathbf{P}^{n-1},\Delta)(\mathbb{A}_{\Q})$ (which is equipped with the subspace topology coming from $H(\mathbb{A}_{\Q})$) and $\pi_{\underline{y}}:Q_{\underline{y}}(\mathbb{A}_{\Q})^{\dag} \rightarrow (\mathbf{P}^{n-1},\Delta)(\mathbb{A}_{\Q})$. (Note that the morphisms $\pi_{\underline{y}}$, introduced in \eqref{morphisms}, define continuous maps 
$$
\begin{array}{ccccc}
\pi_{\underline{y}}&:&Q_{\underline{y}}(\mathbb{A}_\Q) &\rightarrow & H(\mathbb{A}_{\Q})\\
\end{array}
$$ 
which map $Q_{\underline{y}}(\mathbb{A}_{\Q})^{\dag}$ into $(\mathbf{P}^{n-1},\Delta)(\mathbb{A}_{\Q})$.)
\end{definition}
\begin{remark}
From this definition of the measure $\text{\boldmath{$\omega$}}_{(\mathbf{P}^{n-1},\Delta)}$, it follows that its support consists of the (disjoint) union of 
\begin{equation}
 \pi_{\underline{y}}(Q_{\underline{y}}(\mathbb{A}_{\Q})^{\dag}) 
\label{eq:adellen}
\end{equation}
for all $\underline{y}\in \mathbf{Z}_{0}^{n+1}/\{\pm1\}$ with $y_{i}$ squarefree for each $i\in \{0,\ldots,n\}$.
This is a proper subset of $(\mathbf{P}^{n-1}, \Delta)(\mathbb{A}_{\Q})$.
\end{remark}
In order to say something about the convergence of $\mu_{H\leqslant B}^{(\mathbf{P}^{n-1},\Delta)}$, we first define elementary open subsets of $(\mathbf{P}^{n-1},\Delta)(\mathbb{A}_{\Q})$. 

An elementary open subset $W$ of $H(\mathbb{A}_{\Q})$ can be defined as 
$$
W=\prod_{v\in \mbox{Val}(\Q)} W_{v},
$$
such that $W_{v} \subset H(\Q_{v})$ is defined at finitely many finite places as $W_{p}=\text{red}_{M}^{-1}(X_{p})$, where $X_{p} \subset H(\mathbf{Z}/p^{M}\mathbf{Z})$ is a subset of $H(\mathbf{Z}/p^M\mathbf{Z})$ and $\text{red}_{M}: H(\Q_{p}) \rightarrow H(\mathbf{Z}/p^{M}\mathbf{Z})$; $W_{p}=H({\Q}_{p})$ for any other finite place. Furthermore, at the infinite place $v=\infty$, we require 
$W_{\infty}= \bigcap_{i,j} (\lambda_{i,j}x_{i}<x_{j}) \subset H(\mathbf{R})$ fixing one of the coordinates $x_i$ to one. Here, $\lambda_{i,j}\in \mathbf{R}_{>0}$, depending on $i$ and $j$. 

To construct elementary open subsets on $(\mathbf{P}^{n-1},\Delta)(\mathbb{A}_{\Q})$, we can take the intersection with elementary open subsets of $H(\mathbb{A}_{\Q})$.

We will now prove the following theorem.
\begin{theo}\label{equidistribution} For every elementary open subset $U$ of $(\mathbf{P}^{n-1},\Delta)(\mathbb{A}_{\Q})$ it holds that 
$$
\mu_{H\leqslant B}^{(\mathbf{P}^{n-1},\Delta)}(U) \rightarrow \frac{\text{\boldmath{$\omega$}}_{(\mathbf{P}^{n-1},\Delta)}(U)}{\text{\boldmath{$\omega$}}_{(\mathbf{P}^{n-1},\Delta)}((\mathbf{P}^{n-1},\Delta)(\mathbb{A}_\Q))}
$$
as $B$ goes to infinity.
\end{theo}
\begin{proof}
Straightforward calculations show that for each admitted $\underline{y}$, the inverse image $\pi_{\underline{y}}^{-1}(U)$ of an elementary open subset $U$ of $(\mathbf{P}^{n-1},\Delta)(\mathbb{A}_{\Q})$ defines a good open subset of $Q_{\underline{y}}(\mathbb{A}_{\Q})^{\dag}$ .

Now let $U$ be an elementary open subset of $(\mathbf{P}^{n-1},\Delta)(\mathbb{A}_{\Q})$. Recalling \eqref{eq:prob_measure}, the partition of $(\mathbf{P}^{n-1},\Delta)(\Q)^+$ in \eqref{partitie2} and Theorem \ref{orbifold}, we get
\begin{align*}
\mu_{H\leqslant B}^{(\mathbf{P}^{n-1},\Delta)}(U)&=\frac{\#\left\{(u_0:\ldots:u_n)\in (\mathbf{P}^{n-1},\Delta)(\Q)\cap U: \max_{0\leqslant i\leqslant n}|u_{i}|\leqslant B\right\}}{\#(\mathbf{P}^{n-1},\Delta)(\Q)_{\leqslant B}}\\
&\sim \frac{\sum_{\underline{y}}\#\left\{(x_0:\ldots:x_n)\in Q_{\underline{y}}(\Q)^{+} \cap \pi_{\underline{y}}^{-1}(U): \max_{0\leqslant i\leqslant n}|y_i^3x_i^2|\leqslant B\right\}}{\sum_{\underline{y}}\#\left\{(x_0:\ldots:x_n)\in Q_{\underline{y}}(\Q)^{+}: \max_{0\leqslant i\leqslant n}|y_i^3x_i^2|\leqslant B\right\}}. 
\end{align*}
(Here, we used the abbreviated notation $\sum_{\underline{y}}$ to sum over all admitted $\underline{y}\in \mathbf{Z}_0^{n+1}$.)

Combining the fact that rational points on smooth quadrics are equidistributed, the definition of the measure in \eqref{measure} and Theorem \ref{orbifold} enables us to complete the proof. 
\end{proof}